\documentclass[12pt]{iopart}

\usepackage{iopams} 
\usepackage{graphicx}
\usepackage{bm}

\usepackage{subfig}

\usepackage{bbm}
\usepackage{stmaryrd}
\usepackage{mathabx}
\usepackage{amssymb}
\usepackage{graphics}
\usepackage{caption}
\usepackage{graphicx}
\usepackage{mwe}
\usepackage{sidecap}
\usepackage{amsthm}
\newtheorem{thm}{Theorem}
\newtheorem{lemma}{Lemma}
\newtheorem{corollary}{Corollary}
\newtheorem*{remark}{Remark}

\usepackage{cleveref}

\usepackage{etoolbox}

\crefname{equation}{}{}

\patchcmd{\numparts}{\addtocounter{equation}{1}}{\refstepcounter{equation}}{}{}

\newcommand{\N}{\mathbb{N}}

\newcommand{\R}{\mathbb{R}}

\newcommand{\dx}[1][x]{\ensuremath{\ {\mathrm{d}} #1}}
\newcommand{\ds}[1][s]{\ensuremath{\ {\mathrm{d}} #1}}

\theoremstyle{definition}
\newtheorem{exmp}{Example}

\begin{document}

\title[On the required number of electrodes for an inverse coefficient problem]{On the required number of electrodes for uniqueness and convex reformulation in an inverse coefficient problem}

\author{Andrej Brojatsch \& Bastian Harrach}

\address{Institute of Mathematics, Goethe-University Frankfurt, Frankfurt am Main, Germany}
\ead{brojatsch@math.uni-frankfurt.de, harrach@math.uni-frankfurt.de}
\vspace{10pt}
\begin{indented}
\item[]October 2024 
\end{indented}

\begin{abstract}
We introduce a computer-assisted proof for the required number of electrodes for uniqueness and global reconstruction for the inverse Robin transmission problem, where the corrosion function on the boundary of an interior object is to be determined from electrode current-voltage measurements. We consider the shunt electrode model where, in contrast to the standard Neumann boundary condition, the applied electrical current is only partially known. The aim is to determine the corrosion coefficient with a finite number of measurements.

In this paper, we present a numerically verifiable criterion that ensures unique solvability of the inverse problem, given a desired resolution. This allows us to explicitly determine the required number and position of the electrodes. Furthermore, we will present an error estimate for noisy data. By rewriting the problem as a convex optimization problem, our aim is to develop a globally convergent reconstruction algorithm. 
\end{abstract}

%
\vspace{2pc}

\section{Introduction}

In this study, we address the problem of non-destructive impedance-based corrosion detection, which aims to reconstruct an unknown Robin transmission coefficient on a known interior boundary from current-voltage measurements taken at electrodes attached to the outer boundary.

To the best knowledge of the authors, this work presents the first method for explicitly calculating the required number of electrodes for uniqueness and a global reconstruction guarantee in a nonlinear spatially dependent coefficient reconstruction problem. Furthermore we are able to numerically determine an upper bound for the inverse stability constant and hereby provide error estimates to the solution.

Traditionally, practitioners approach such reconstruction problems using regularized data fitting methods. The non-convexity of the residual function presents a significant challenge for nonlinear inverse problems, which is why, in general, only a few global reconstruction algorithms are known. Previous works \cite{harrach2022solving, Harrach_2020uniqueness} introduced a novel approach to overcoming non-convexity and ensuring global convergence in an idealized setting with standard Neumann boundary conditions. In the present paper, we translate these results to the shunt electrode model, which extends the standard Neumann boundary problem to a realistic electrode model, where the applied electrical current is only partially known.  We establish a criterion for determining whether a given number of electrodes is sufficient for unique solvability of the inverse problem and provide an equivalent reformulation by rewriting the problem as a semidefinite program. Our results explicitly characterize the minimum number of electrodes required to guarantee the reconstruction of an unknown Robin parameter with a given resolution.

By evaluating a finite number of forward solutions, and their derivatives, the criterion can be numerically verified for a given desired resolution, so the proof for uniqueness and global reconstruction is computer-assisted. Furthermore, our criterion provides explicit error estimates for noisy data. We also refine previous uniqueness results \cite{harrach2022solving} to obtain sharper stability assumptions.

This work contributes to the broader framework of the Calderón problem \cite{calderon2006inverse}, a classical inverse problem concerned with reconstructing information about the interior of a domain information from measurements collected at its boundary. Electrical impedance tomography (EIT) is a well-known application of this problem \cite{adler2021electrical}. In EIT, one seeks to reconstructs material properties from a finite number of current-voltage measurements. While global reconstruction results are well-established for the case of infinitely many measurements \cite{caro2016global, krupchyk2016calderon}, results for the practically relevant scenario of finitely many measurements remain limited \cite{alberti2019calderon, alessandrini2005lipschitz, harrach2019uniqueness, harrach2023calderon}. The authors of \cite{ruland2018lipschitz} address the fractional case for finitely many measurements. Nevertheless, these findings are of theoretical nature and are not directly applicable in practical settings. 

The Robin problem considered in \cite{harrach2022solving} is a special case where reconstruction is restricted to a small subdomain. Classical theoretical tools, such as Runge approximation and localized potentials, hold in general for this case \cite{harrach2019global}. Previous uniqueness results for the Robin case primarily focus on the infinite-dimensional setting, assuming infinite resolution and infinitely many measurements \cite{harrach2019global}. In contrast, \cite{ Harrach_2020uniqueness,harrach2022solving} address the practical scenario of a finite resolution and finitely many measurements, where a measurement corresponds to measuring the exact Dirichlet data from applied Neumann boundary values. In \cite{harrach2019global}, the authors have established Lipschitz stability results for cases involving finitely many unknowns and infinitely many measurements.

In this work, we consider the Robin problem under a realistic electrode model, where electrodes are attached to the domain boundary \cite{cheney1999electrical}. The most common electrode model is the complete electrode model. However, we focus on the shunt model, which simplifies the setup by neglecting electrode-body impedance.  
Foundational theoretical work on electrode models can be found in \cite{ hanke2011justification, harrach2015interpolation, hyvonen2017smoothened, somersalo1992existence}. 
Theoretical attempts to approximate idealized models using electrode models have been made in \cite{harrach2019uniqueness, hauptmann2017approximation, hyvonen2004complete, hyvonen2009approximating}. However, in EIT little is known about characterizing the number of electrodes required to achieve a desired resolution.

By extending the theoretical results from \cite{harrach2022solving} to an electrode model, we exploit the fact that the Neumann-Dirichlet operator already has a finite structure. The theoretical result \cite[Theorem 1]{harrach2022solving} translates into an explicit criterion that can be numerically verified by computing a constant $\lambda$ that was found to be the inverse stability constant. Due to the inherent ill-posedness of the nonlinear inverse problem, $1/\lambda$ is expected to grow exponentially as stated in \cite{alberti2022infinite}. Our numerical experiments in Section \ref{sec:experiments} confirm that $\lambda$ reaches machine precision even for low resolutions. Nevertheless, we observe stable and global convergence with our equivalent reformulation. Additionally, we refine the uniqueness proof from \cite[Theorem 1]{harrach2022solving} to improve the stability constant $\lambda$.

\begin{figure}[h]
\begin{center}
\includegraphics[width=0.5\linewidth]{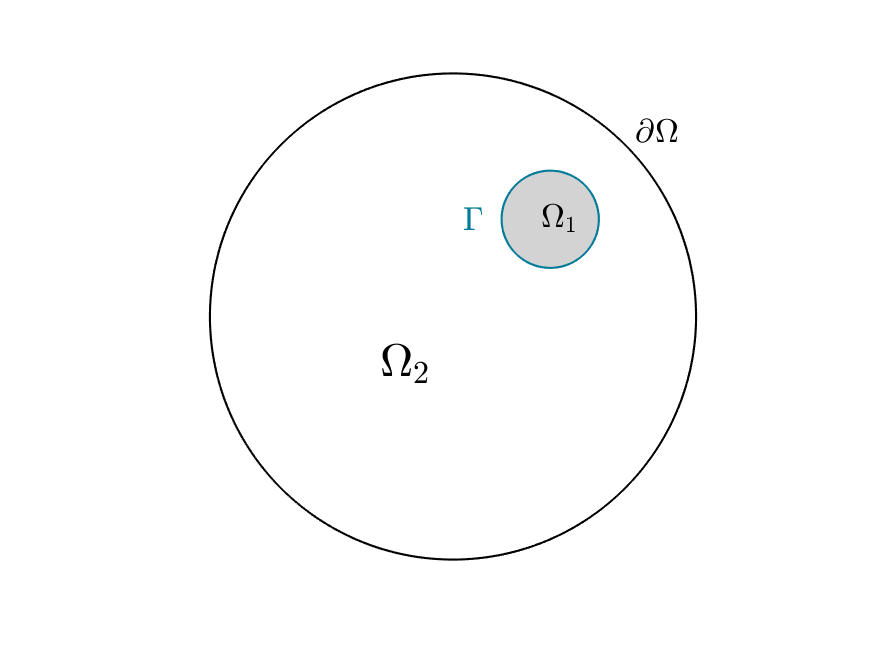}\\
\caption{The domain $\Omega=\Omega_1 \cup \Omega_2$.}\label{fig1}
\end{center}
\end{figure}

\section{Main results}
\label{sec:main}

Let $\Omega\subset \R^d$ be an open, bounded domain with Lipschitz boundary $\partial \Omega$, $\Omega_1$ be an open subset of $\Omega$ with Lipschitz boundary $\Gamma:= \partial \Omega_1$, such that $\Gamma$ is entirely contained within $\Omega$. Define  $\Omega_2:=\Omega \setminus \Omega_1$, assuming $\Omega_2$ is connected (see Figure \ref{fig1}). Additionally, let the electrodes $\mathcal{E}_1,\ldots \mathcal{E}_m$ be measurable subsets of $ \partial \Omega$ with nonzero measure.  

Given a piecewise-constant conductivity, defined as $\sigma=\sigma_1\chi_{\Omega_1}+\sigma_2\chi_{\Omega_2}$, where $\sigma_1,\sigma_2>0$, and a Robin transmission coefficient $\gamma \in L^\infty_+(\Gamma)$, we consider the following problem with applied currents $I_k\in \R$ for $k=1,\ldots,m$
\begin{equation}\label{robineqn}
 \left\{
 \begin{array}{rcll}
-\nabla \cdot \left( \sigma \nabla u \right) &=& 0  &\ \ \textrm{on} \ \ \Omega_1 \cup \Omega_2, \\
\int_{\mathcal{E}_k} \sigma \partial_\nu u \, \mathrm{d}s &=& I_k  &\ \ \textrm{for} \ \ k = 1, \ldots, m, \\
\sigma \partial_\nu u &=& 0 &\ \ \textrm{on} \ \ \partial \Omega \setminus \bigcup_{k=1,\ldots,m} \mathcal{E}_k, \\
\left\llbracket u \right\rrbracket &=& 0 &\ \ \textrm{on} \ \ \Gamma, \\
\left\llbracket \sigma \partial_\nu u \right\rrbracket &=& \gamma u &\ \ \textrm{on} \ \ \Gamma, \\
u \big|_{\mathcal{E}_k} &=& \textrm{const.} =: U_k &\ \ \textrm{for} \ \ k=1,\ldots,m,
\end{array}
 \right.
 \end{equation}
where $\nu$ is the unit normal vector on $\Gamma$ or $\partial \Omega$ pointing outward of $\Omega_1$ or $\Omega$,
$\llbracket \varphi \rrbracket  := \textrm{trace}\left(\varphi \big|_{\Omega_2}\right)-\textrm{trace}\left(\varphi \big|_{\Omega_1}\right)$and $\llbracket \sigma \partial_\nu \varphi \rrbracket_\Gamma = \partial_\nu \left(\sigma_2 \varphi\big|_{\Omega_2}\right)-\partial_\nu \left( \sigma_1\varphi\big|_{ \Omega_1}\right)$. 

A known current $I_k$ is applied to the electrode $\mathcal{E}_k$ for $k=1,...,m$, and the necessary  voltages $U_k$ are measured on the same electrodes. This setup is referred to as the shunt model. In contrast to the complete electrode model, the shunt model represents an idealized scenario where perfect conduction between the body and the electrodes is assumed. The corresponding variational formulation of the shunt electrode model is given by
\begin{equation*}\label{variational}
b(u,v):=\int_\Omega  \nabla u \nabla v  \dx + \int_\Gamma \gamma u~ v \ds = \sum_{k=1}^m I_k \left( v \big|_{\mathcal{E}_k} \right) \qquad \textrm{for all}~ v \in H_{\square}^1(\Omega),
\end{equation*}
with $H_\square^1(\Omega) = \lbrace v \in H^1(\Omega):~  v \big|_{\mathcal{E}_k} = \textrm{const.} \quad \textrm{for}~  k=1,\ldots,m \rbrace$.  Since $b$ is continuous and coercive and $H_\square^1(\Omega)$ is a closed subspace of $H^1(\Omega)$ the non-linear partial differential Equation \cref{robineqn} is uniquely solvable for all $\gamma\in L^\infty_+(\Gamma)$ by Lax--Milgram Theorem (set forth in \cite{somersalo1992existence}). For a piecewise constant coefficient function $\gamma \in L^\infty_+(\Gamma)$ on a partition $\Gamma_1,\ldots,\Gamma_n$, $n\geq 2$, where
\[ \Gamma_1,\ldots,\Gamma_n, ~ \textrm{are pairwise disjoint, measurable and}~ \bigcup\limits_{j=1}^n\Gamma_j=\Gamma,\] 
one can define a forward map  $\mathcal{F}$ to model the measurements. $\mathcal{F}\colon \R^n_+ \to \mathbb{S}^m_+ $ is given by
\begin{equation}\label{NtD}
\mathcal{F}(\gamma)I = U, \quad \textrm{where} ~u_\gamma^{(I)}~ \textrm{solves \cref{robineqn} and $u_\gamma^{(I)}\big|_{\mathcal{E}_j}= U_j$ for $j=1,\ldots m$ }. 
\end{equation}
Here $\mathbb{S}^m_+ $ denotes the space of the symmetric positive definite matrices and we identify the vector $\gamma\in \R^n_+$ with the piecewise constant function $\left(\sum\limits_{j=1}^n \gamma_j \chi_{\Gamma_j} \right)\in L^\infty_+(\Gamma)$. Note that $\mathcal{F}(\gamma) \in \mathbb{S}^m_+ $ since
\[ I^T \mathcal{F}(\gamma)J =b(u_\gamma^{(I)},u_\gamma^{(J)}) \qquad \textrm{for all}~ I,J \in  \R^m \]
and $b$ is symmetric and coercive. 
\begin{figure}[h]
\begin{center}
\includegraphics[width=0.5\linewidth]{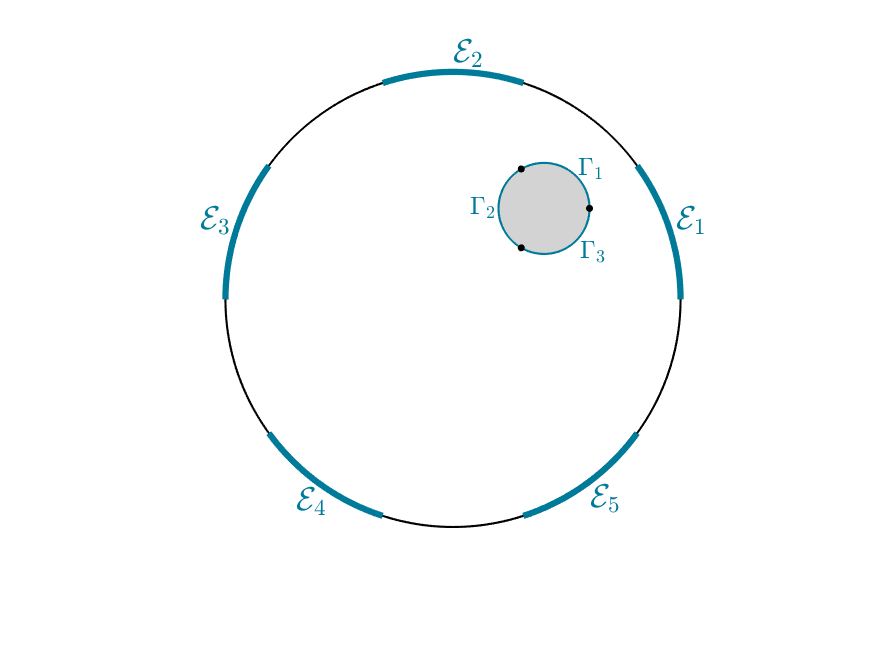}\\
\caption{Resolution dimension $n=3$ and $m=5$ electrodes.}\label{fig44}
\end{center}
\end{figure}

The physical interpretation of $\mathcal{F}(\gamma)$ is that it represents the measured voltage $U\in \R^m$ at the electrodes in response to the applied currents $I\in \R^m$. Specifically, the $(jk)$-th component of $\mathcal{F}(\gamma) $ corresponds to the voltage measured at electrode $\mathcal{E}_j$ when a current is applied to electrode $\mathcal{E}_k$. Hence, we refer to $\mathcal{F}$ as the measurement operator and to $\mathcal{F}(\gamma)$ as the current-voltage measurements. The resulting inverse problem is formulated as follows:
\begin{center}
\begin{equation}\label{inverseproblem}
\hspace{-9ex} \textit{Reconstruct the corrosion parameter}~\gamma ~ \textit{from the measurements}~ \hat Y=:\mathcal{F}(\hat \gamma). 
 \end{equation}
  \end{center}
Even if the unique solvability of the inverse problem \cref{inverseproblem} can be proven, major challenges arise due to its ill-posedness and the non-convexity of the natural data fitting approach. In Figure \ref{fig2444}, we apply the generic \texttt{MATLAB} solver \texttt{lsqnonlin} to minimize the least-squares residual function
  \begin{center}
 \[ \Vert \mathcal{F}(\gamma)-\hat Y\Vert_F^2 \to \min ! \]
   \end{center}
using the first order trust-region-reflective algorithm to minimize the least-squares residual function and demonstrate the difficulty of global reconstruction. Here, we consider the setup of a small circle inside a larger circle with an equidistant partition $\Gamma=\Gamma_1 \cup \ldots  \cup \Gamma_n$ and uniformly placed electrodes $\mathcal{E}_1,\ldots,\mathcal{E}_m$ along the boundary $\partial\Omega$, as in Figure \ref{fig44}, with $n=2$, $m=4$ and bounds $a=1$ and $b=3$. 
\begin{figure}[h]
    \centering
  {{\includegraphics[width=0.46\textwidth]{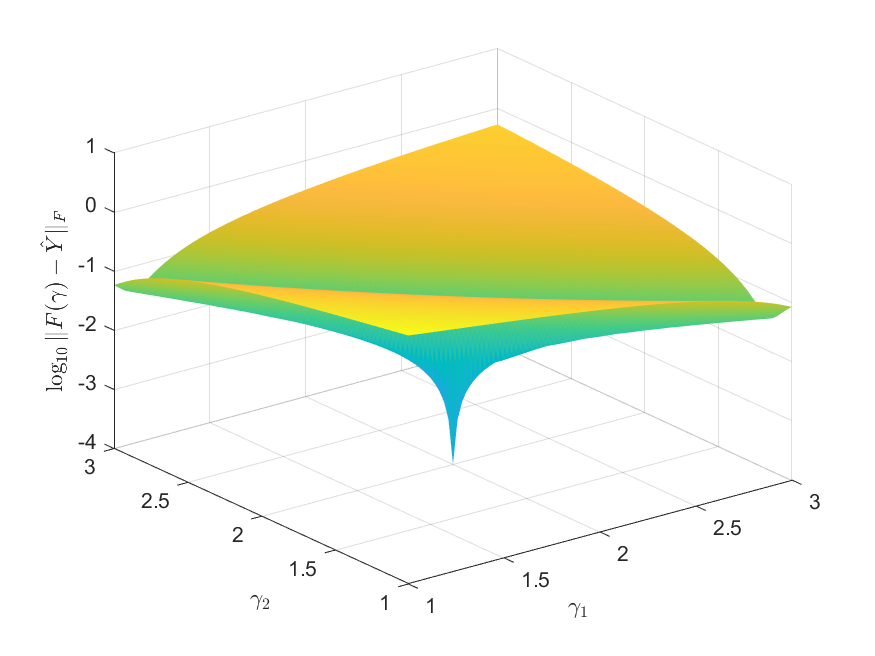} }}%
    \qquad
   {{\includegraphics[width=0.46\textwidth]{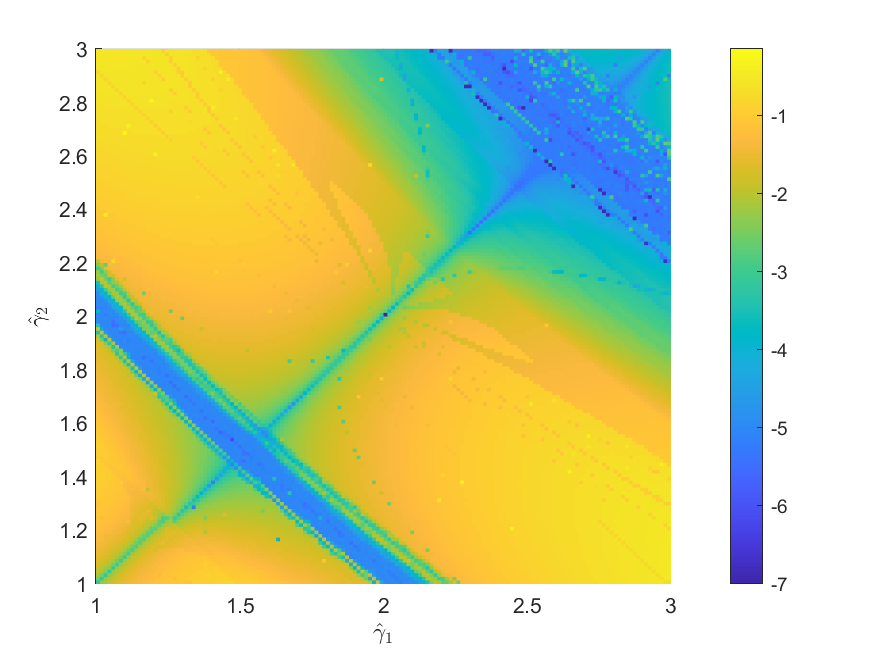} }}%
    \caption{Plot of the residual function with $\hat \gamma =(2, 2)$ and the logarithmic error $\log_{10}(\Vert \gamma^{(N)}- \hat{\gamma}\Vert_2)$ for the least squares data fitting approach with varying corrosion parameters $\hat \gamma=(\hat \gamma_1,\hat \gamma_2)$. The initial value is set to $\gamma^{(0)} =(2, 2)$. The maximum error $\Vert \gamma^{(N)}- \hat{\gamma}\Vert_2  \approx 0.72$ occurs at $\hat{\gamma}=(1.09,2.68)$.}
    \label{fig2444}%
\end{figure}

To overcome these difficulties, we establish a criterion. For this, one must choose a priori bounds $a<\gamma_i<b$ for $i=1,\ldots,n$. The criterion ensures that, for a given number of electrodes $m$:
\begin{itemize}
\item The inverse problem \cref{inverseproblem} is uniquely solvable.
\item Provide an explicit Lipschitz stability constant.
\item There is a globally convergent reconstruction algorithm.
\end{itemize}

\begin{figure}[h]
\begin{center}
\includegraphics[width=0.46\textwidth]{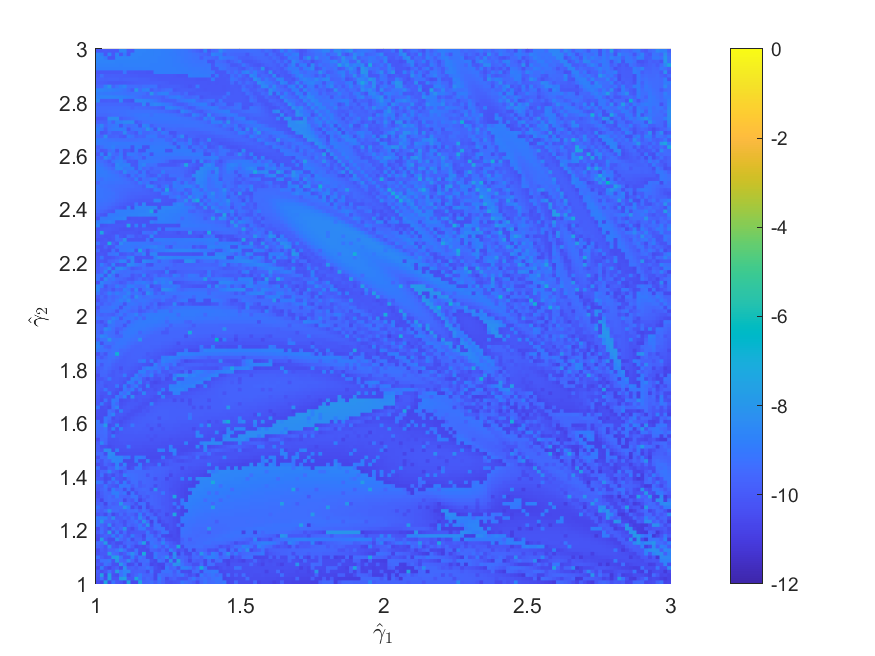}
\caption{Plot of the equivalent approach for varying corrosion parameters $\hat \gamma=(\hat \gamma_1,\hat \gamma_2)$, with the initial value set to $\gamma^{(0)} =(2, 2)$. The maximum error is $\Vert \gamma^{(N)}- \hat{\gamma}\Vert_2\leq 5.5\cdot 10^{-7}$.}\label{cvxconv}
\end{center}
\end{figure}

To compare the least squares data fitting approach shown in Figure \ref{fig2444} with our convex reformulation, which does not require an initial value, we we initialize $\gamma^{(0)} =(2, 2)$ and systematically vary the true corrosion parameter $\hat \gamma=(\hat \gamma_1,\hat \gamma_2)$ within the square $[a,b]^2$. Figure \ref{cvxconv} illustrates that convergence can be directly improved through the equivalent convex reformulation, as proposed in Theorem \ref{mainsemidefishunt}. The key idea behind formulating such a criterion for uniqueness and global convergence is based on monotonicity and convexity relations, as presented in \cite{harrach2022solving}.

To formulate the main results, we first introduce some notation. For $x,y\in \R^n$, the relation $ x\leq y$ is understood pointwise, meaning $ x\leq y$ if and only if $ x_i \leq y_i $ for $i=1,\ldots,n$. Similarly, for $A,B\in \mathbb{S}^m$, the relation $A	\preceq B $ is understood in the Loewner sense, i.e.,  $A	\preceq B $ if and only if $\left(  B-A \right)$ is positive semidefinite. 

Furthermore let $e_j\in \R^n$ denote $j$-th unit vector, and define $e_j'\colon=\mathbbm{1}-e_j\in \R^n$ as the negated unit vector, which has a zero at the $j$-th component and ones in all other components.

\begin{thm}\label{uniqshunt}
For a resolution of dimension $n\in \N$, $m\in \N$ fixed electrodes and $n,m\geq 2$ let $\mathcal{F}\colon \R^n_+ \to \mathbb{S}_m $ be the measurement operator $\mathcal{F}(\gamma)I = U$ as in Equation \cref{NtD}. Furthermore let $a,b\in \R$ with $0<a<b$. If 
\begin{equation}\label{criterion1}
\mathcal{F}'\left(z_{j,k}\right)d_j\npreceq 0 \quad \textrm{for all}\quad k\in \lbrace 2,\ldots,K\rbrace,~ j\in \lbrace 1,\ldots,n \rbrace, 
\end{equation}
where 
\[z_{j,k}:=\frac{a}{2}e_j' + \left(a+k\frac{a}{4}\right)  e_j\in \R^n_+, ~ d_j:=\frac{2b-a}{a}e_j' - \frac{1}{2}   e_j\in \R^n
\]
and  $K:=\max\left(\lceil\frac{4b}{a}\rceil-3,2\right)$, then the following holds:
\begin{itemize}
\item[(a)] $\mathcal{F}'(\gamma)\in \mathcal{L}\left( \R^n, \mathbb{S}_m \right)$ is injective for all $\gamma\in [a,b]^n$ and for all $0\neq d\in \R^n$,
\[ \frac{\Vert\mathcal{F}'(\gamma)d\Vert_2}{\Vert d\Vert_\infty}\geq \lambda := \min_{\mathop{j=1,\ldots,n}\limits_{k=2,\ldots,K}}\lambda_{\max} \left(\mathcal{F}'(z_{j,k})d_j\right)>0. \]
\item[(b)] $\mathcal{F}\colon \R^n_+ \to \mathbb{S}^m $ is injective and
\[ \Vert \mathcal{F}(\gamma_1)-\mathcal{F}(\gamma_2) \Vert_2 \geq \lambda \Vert \gamma_1-\gamma_2 \Vert_\infty \]
for all $\gamma_1,\gamma_2\in [a,b]^n$ i.e. the inverse problem
\begin{center}
$ \textrm{determine $\hat \gamma\in [a,b]^n$ from the knowledge of $ \mathcal{F}(\hat{\gamma})$} $
\end{center}
is uniquely solvable.
\end{itemize}
\end{thm}
Criterion \cref{criterion1} allows us to decide whether a number of electrodes is sufficient for unique solvability of the inverse problem \cref{inverseproblem}. The evaluation of the criterion already provides an inverse Lipschitz stability constant. The results in Theorem \ref{uniqshunt} generalize the uniqueness result in Theorem 1 in \cite{harrach2022solving} to tighter assumptions. For the proof of Theorem \ref{uniqshunt} in Section \ref{section2} and Section \ref{section3}, we will generalize Lemma 1 and Lemma 2 in \cite{harrach2022solving}.

For a fixed number of electrodes $m$ and a resolution of dimension $n$, one can explicitly check criterion \cref{criterion1}, and therefore decide whether the number and position of the electrodes is sufficient for unique solvability and Lipschitz stability of the inverse problem. Our second result provides a criterion for a convex reformulation of the considered inverse problem and gives an error estimate for noisy data. Note that the assumptions of the following theorem imply those of Theorem \ref{uniqshunt}, as explained in Section \ref{section3}. 

\begin{thm} \label{mainsemidefishunt}
Let $n,m\geq 2$, $\mathcal{F}\colon \R^n_+ \to \mathbb{S}_m $ be the measurement operator $\mathcal{F}(\gamma)I = U$ and $0<a<b$. For the assumptions 
\[
\hspace{-9ex} z_{j,k}:=\frac{a}{2}e_j' + \left(a+k\frac{a}{4(n-1)}\right)  e_j\in \R^n_+, ~ d_j:=\frac{2b-a}{a}(n-1)e_j' - \frac{1}{2}   e_j\in \R^n
\] 
and  $K:=\max\left(\lceil\frac{4(n-1)b}{a}\rceil-4n-3,2\right)$ it holds: \\ \\
If 
\begin{equation}\label{criterion2}
\mathcal{F}'\left(z_{j,k}\right)d_j \npreceq 0 \quad \textrm{for all}\quad k\in \lbrace 2,\ldots,K\rbrace,~ j\in \lbrace 1,\ldots,n \rbrace, 
\end{equation}
then the following holds additionaly to the asserations of Theorem \ref{uniqshunt}:
\begin{itemize}
\item[(a)] The inverse problem
\begin{center}
determine $\hat \gamma\in [a,b]^n$ from the knowledge of $\hat Y = \mathcal{F}(\hat{\gamma})$
\end{center}
is uniquely solvable and $\hat \gamma$ is the unique minimizer of the convex optimization problem:
\begin{equation}\label{semidefiniteform}
\texttt{minimize} ~ \sum_{i=1}^n \gamma_i~ \textrm{subject to} ~ \gamma \in [a,b]^n, ~ \mathcal{F}(\gamma) \preceq \hat Y.  
\end{equation}
\item[(b)]
For $\hat{\gamma}\in [a,b]^n$, $\delta>0$ and $Y^\delta \in \mathbb{S}^m $, with $\Vert \hat Y- Y^\delta \Vert_2 \leq \delta$ the convex problem:
\[
\texttt{minimize} ~ \sum_{i=1}^n \gamma_i~ \textrm{subject to} ~ \gamma \in [a,b]^n, ~ \mathcal{F}(\gamma) \preceq  Y^\delta+\delta I  
\]
possesses a minimum, and every such minimum $\gamma^\delta$ fulfills
\[ \Vert \hat{\gamma}- \gamma^\delta \Vert_\infty \leq \frac{2\delta(n-1)}{\lambda},\]
with $\lambda :=\min_{j,k}\lambda_{\max} \left(\mathcal{F}'(z_{j,k})d_k \right)$.
\end{itemize}
\end{thm}

Theorem \ref{mainsemidefishunt} follows directly  from \cite[Theorem 1]{harrach2022solving} using the monotonicity and convexity relations established in Lemma \ref{monoton}.  The core idea of the proof in \cite{harrach2022solving} is based on a converse monotonicity result. It is shown that if $\mathcal{F}$ satisfies condition \cref{criterion2}, then
\[ \mathcal{F}(y)\preceq \mathcal{F}(x) \quad \textrm{implies}\quad \sum_{i=1}^n(y_i-x_i)>0 \qquad \textrm{for all}~x,y\in \R^n_+,~x\neq y.  \]
This provides a reformulation of the inverse problem \cref{inverseproblem} as a convex semi-definite optimization problem. But here the constant $C=(n-1)$, which extends the condition of Theorem \ref{uniqshunt}, plays a substantial role.

\begin{lemma} \label{monoton}
The forward map $\mathcal{F}\colon \R^n_+ \to \mathbb{S}^m $ is infinitely many times differentiable and $\mathcal{F}'\colon \R^n_+\to \mathcal{L}\left(\R^n, \mathbb{S}^m \right)$ is  given  by
\begin{equation*}
 I^T \left( \mathcal{F}'(\gamma)\delta \right) J = -\sum_{i=1}^n \delta_i \int_{\Gamma_i}  u_\gamma^{(I)}u_\gamma^{(J)} \dx.
\end{equation*}
Moreover  $\mathcal{F}$ is monotonically non-increasing and convex, i.e. it holds 
\begin{equation*}
\begin{array}{rcll}
\mathcal{F}'(\gamma)\delta &\preceq& 0 & \quad \textrm{for all } \gamma \in \mathbb{R}^n_+, \, 0 \leq \delta \in \mathbb{R}^n, \\[8pt]
\mathcal{F}(\gamma) - \mathcal{F}(\gamma^{(0)}) &\succeq& \mathcal{F}'(\gamma^{(0)})\left(\gamma - \gamma^{(0)}\right) & \quad \textrm{for all } \gamma, \gamma^{(0)} \in \mathbb{R}^n_+.
\end{array}
\end{equation*}
\end{lemma}
\begin{proof}
This follows directly from Lemma 2 and Corollary 1 in \cite{harrach2021introduction}. 
\end{proof}
Monotonicity and convexity can also be written as
\[0< \gamma^{(0)}\leq \gamma^{(1)} \qquad \textrm{implies} \qquad \mathcal{F}\left( \gamma^{(0)} \right) 	\succeq \mathcal{F}\left( \gamma^{(1)} \right), 
\]
and for all $ \gamma^{(0)},\gamma^{(1)}\in \R^n_+$, $t\in[0,1]$
\[ \mathcal{F}\left((1-t) \gamma^{(0)} +t \gamma^{(1)} \right)\preceq (1-t) \mathcal{F}\left( \gamma^{(0)} \right) +t \mathcal{F}\left( \gamma^{(1)} \right) . \]
The equivalence to the differential characterization in Lemma \ref{monoton} is shown in \cite[Lemma 2]{harrach2021introduction}.

\subsection{A criterion for uniqueness and Lipschitz stability}\label{section2} To prove Theorem \ref{uniqshunt} we first derive a sufficient criterion for unique solvability of the inverse problem that involves the directional derivatives in arbitary points $x\in [a,b]^n$.
\begin{lemma}\label{lem11}
Let $m,n\in \N$, $n,m\geq 2$ and $F\colon \R^n_+ \to \mathbb{S}_m $ be continuously differentiable, monotonically non-increasing and convex.
\begin{itemize}
\item[(a)] If for some $x\in  \R^n_+$
\[  F'\left(x \right)\left( e_j'-e_j \right)\npreceq 0 \quad \textrm{for all}\quad j\in \lbrace 1,\ldots,n \rbrace,\]
then $F(x)'\in \mathcal{L}\left( \R^n, \mathbb{S}_m \right)$ is injective for all $x\in [a,b]^n$. Moreover, for all $0\neq d\in \R^n$,
\begin{equation}\label{lemeq}
\frac{\Vert F(x)'d\Vert_2}{\Vert d\Vert_\infty}\geq \lambda_x := \min_{j=1,\ldots,n}\lambda_{\max} \left(F'\left(x \right)\left( e_j'-e_j \right)\right)>0.
\end{equation} 
\item[(b)] If $x,y\in [a,b]^n$ fulfill \cref{lemeq}, then
\[ \Vert F(x)-F(y) \Vert_2 \geq \lambda \Vert x-y \Vert_\infty \]
with $\lambda:=\min\lbrace \lambda_x,\lambda_y \rbrace>0$.
\end{itemize}
\end{lemma}

\begin{proof}
This proof is completely analogous to \cite[Lemma 1]{harrach2022solving}. But since it is short and simple, we state it for the sake of completeness.
\begin{enumerate}
\item[(a)] Let $d\in \R^n$ with $\Vert d\Vert_\infty=1$. Then at least one of the entries of $d$ must be either $1$ or $-1$, so that there exists $j\in \lbrace 1,\ldots,n \rbrace$ with either
\[ d\leq e_j'-e_j \quad \textrm{or} \quad -d\leq e_j'-e_j.  \]
Since $F$ is monotonically non-increasing we obtain
\[ F'\left(x \right)d \succeq F'\left(x \right)\left( e_j'-e_j \right) \quad \textrm{or} \quad F'\left(x \right)(-d) \succeq F'\left(x \right)\left( e_j'-e_j \right) \]
for every $x\in [a,b]^n$ so that $F'\left(x \right)d$ has at least one positive eigenvalue that is larger than $\lambda_{\max} \left(F'\left(x \right)\left( e_j'-e_j \right)\right)$ or smaller than $-\lambda_\textrm{max} \left(F'\left(x \right)\left( e_j'-e_j \right)\right)$ for every $j=1,\ldots,n$. This proves $(a)$.
\item[(b)] With the same argument as in $(a)$ we have for $x\neq y$ either
 \[ \frac{x-y}{\Vert x-y \Vert_\infty}\leq e_j'-e_j \quad \textrm{or} \quad \frac{y-x}{\Vert x-y \Vert_\infty}\leq e_j'-e_j.  \]
 Hence by monotonicity and convexity we have that either
\begin{equation*}
\begin{array}{rl}
\displaystyle \frac{F(x)-F(y)}{\Vert x-y \Vert_\infty} &\succeq F'\left(y \right)\frac{x-y}{\Vert x-y \Vert_\infty} \succeq F'\left(y \right)\left( e_j'-e_j \right) \quad \textrm{or} \\[12pt]
\displaystyle \frac{F(y)-F(x)}{\Vert x-y \Vert_\infty} &\succeq F'\left(x \right)  \frac{y-x}{\Vert x-y \Vert_\infty} \succeq F'\left(x \right)\left( e_j'-e_j \right),
\end{array}
\end{equation*}
 which shows that
 \[ \Vert F(x)-F(y) \Vert_2 \geq \lambda_y \Vert x-y \Vert_\infty \quad \textrm{or} \quad \Vert F(x)-F(y) \Vert_2 \geq \lambda_x \Vert x-y \Vert_\infty,  \]
 so that (b) is proven.
\end{enumerate}
\end{proof}
For inverse coefficient problems with finitely many measurements, the assumption
\begin{equation} \label{localizedpotentialproperty}
\mathcal{F}'\left(x \right)\left(  e_j'-e_j \right)\npreceq 0 
\end{equation} 
can be interpreted in the sense of localized potentials, since
\[ I^T\mathcal{F}'\left(x \right)\left( e_j'-e_j \right)I= \int_{\Gamma_i}  \left( u_\gamma^{(I)}\right)^2 \dx - \sum_{j=1 \atop j \neq i} \int_{\Gamma_j}  \left( u_\gamma^{(I)}\right)^2 \dx  \]
for all $I\in \R^m$. This expression is positive if one can ensure that the solution is large on some part of $\Gamma$ and small on the other parts.
For some applications such as the idealized Robin problem (\cite[Section 3]{harrach2022solving}) it follows from \cite[Theorem 2(b)]{harrach2022solving} that these assumptions hold for all $\gamma\in [a,b]^n$ when enough measurements are being used, i.e. if $m$ is suffenciently large. For the shunt model \cref{robineqn} it remains to be shown that potentials can be localized with a sufficient number of electrodes, such that the assumptions of Theorem \ref{uniqshunt} are fulfilled. In \cite{hyvonen2009approximating} the authors provide  approximation results of the continuum
model by the complete electrode model. These should hold analogously for the shunt electrode model in the Robin case.

\subsection{Checking the assumption with only finitely many tests} \label{section3}
To practically determine how many measurements are enough, we will now rewrite the assumption so that they require evaluations of $\mathcal{F}'$ for only finitely many points. We finalize the proof of Theorem \ref{uniqshunt} by generalizing Lemma 2 in \cite{harrach2022solving}. Therefore we argue that if one checks the localized potential property \cref{localizedpotentialproperty} for finitely many points, then the assertions of Lemma \ref{lem11} directly follow for the whole required domain. The key idea is that we can replace the evaluation points of $F'$ by using the convexity property
\[ F'(x)(y-x)\succeq F(y)-F(x)\succeq F'(y)(y-x) \]
and the direction of the derivative by the monotonicity property
\[ F'(x)d_2 \succeq F'(x)d_1 \qquad \textrm{for}~ d_1 \geq d_2. \]
\begin{lemma} \label{lemma2}
Let $F: \mathbb{R}_{+}^n \rightarrow \mathbb{S}_m, n, m \geq 2$, be continuously differentiable, convex and monotonically non-increasing, $b \geq a>0$, $C>0$, $ d_j:=\frac{2b-a}{a}Ce_j' - \frac{1}{2}   e_j\in \R^n
$. Then, for all $j \in\{1, \ldots, n\}$ and $x \in[a, b]^n$, there exists $t \in$ $\left[a+\frac{a}{2C}, b+\frac{a}{2C}\right] \subset \mathbb{R}$, so that for all $0 \leq \delta \leq \frac{a}{4C}$,
\[
F^{\prime}(x)\left( Ce_j^{\prime}-e_j\right) \succeq F^{\prime}\left(\frac{a}{2} e_j^{\prime}+(t-\delta) e_j\right) d_j .
\]
\end{lemma}

\begin{proof}
Let $j \in\{1, \ldots, n\}$ and $x \in[a, b]^n$. We define $t:=x_j+\frac{a}{2c}$. Then for all $0 \leq \delta \leq \frac{a}{4c}$
\[
\begin{array}{rl}
C e_j^{\prime}-e_j &= \displaystyle \frac{2C}{a}\left(\frac{a}{2} e_j^{\prime}+(x_j - t)e_j\right) \leq \frac{2C}{a}\left(x - \frac{a}{2} e_j^{\prime} - t e_j\right) \\[10pt]
& \leq \displaystyle \frac{2C}{a}\left(x - \left(\frac{a}{2} e_j^{\prime} + (t - \delta)e_j\right)\right)
\end{array}
\]
and
\[
\begin{array}{rl}
& \displaystyle \frac{2C}{a}\left(x - \left(\frac{a}{2} e_j^{\prime} + (t - \delta)e_j\right)\right) \leq \displaystyle \frac{2C}{a}\left(\left(b - \frac{a}{2}\right)e_j^{\prime} + \left(x_j - t + \delta\right)e_j\right) \\[10pt]
& \quad = \displaystyle \frac{2b - a}{a}(n - 1)e_j^{\prime} + \frac{2(n - 1)}{a}\left(-\frac{a}{2C} + \delta\right)e_j \\
& \quad \leq \displaystyle \frac{2b - a}{a}C e_j^{\prime} - \frac{1}{2} e_j = d_j.
\end{array}
\]
so that we obtain from monotonicity and convexity

\[
\begin{array}{rl}
F^{\prime}(x)&\hspace{-2ex}\left(C e_j^{\prime} - e_j\right)  \succeq \displaystyle \frac{2C}{a} F^{\prime}(x) \left(x - \left(\frac{a}{2} e_j^{\prime} + (t - \delta)e_j\right)\right)
\\[5pt]
& \succeq \displaystyle \frac{2C}{a} \left(F(x) - F\left(\frac{a}{2} e_j^{\prime} + (t - \delta)e_j\right)\right) \\[10pt]
& \succeq \displaystyle \frac{2C}{a} F^{\prime}\left(\frac{a}{2} e_j^{\prime} + (t - \delta)e_j\right)\left(x - \left(\frac{a}{2} e_j^{\prime} + (t - \delta)e_j\right)\right) \\[10pt]
& \succeq F^{\prime}\left(\frac{a}{2} e_j^{\prime} + (t - \delta)e_j\right) d_j.
\end{array}
\]
\end{proof}

\begin{corollary}\label{coro2}
Let $C>0$ and $j \in\{1, \ldots, n\}$. Choose $K\geq 2$ so that $a+K \frac{a}{4C} \geq b+\frac{a}{4C}$. If
\[
\left(F^{\prime}\left(z_{j, k}\right)\left(d_j\right)\right)\npreceq 0 \qquad \textrm{for all}~ k\in \lbrace 2,\ldots,K\rbrace,
\]
with $z_{j,k}=\frac{a}{2}e_j' + \left(a+k\frac{a}{4C}\right)  e_j\in \R^n_+, ~ d_j=\frac{2b-a}{a}Ce_j' - \frac{1}{2}   e_j\in \R^n
$, then
\[
F^{\prime}(x)\left( Ce_j^{\prime}-e_j\right)  \npreceq 0 \qquad \textrm{for all}~ x\in [a,b]^n.
\]
\end{corollary}
\begin{proof}
Since $a+K \frac{a}{4C} \geq b+\frac{a}{4C}$, we have for every $t \in\left[a+\frac{a}{2C}, b+\frac{a}{2C}\right]$ that $\left(t-\left(a+K \frac{a}{4C}\right)\right)\leq \frac{a}{4C} $ and $0 \leq \left(t-\left(a+ 2\frac{a}{4C}\right)\right)$. Therefore there exists $k \in \{2, \ldots, K\}$ so that
\[
\delta:=t-\left(a+k \frac{a}{4C}\right) \quad \textrm{ fulfills } \quad 0 \leq \delta \leq \frac{a}{4C},
\]
since one will certainly land in $\lbrack  0 , \frac{a}{4C} \rbrack$ with step size $\left(\frac{a}{4C}\right) $.
The asseration follows from Lemma \ref{lemma2}.
\end{proof}

\begin{remark}
Note that 
\[a+K \frac{a}{4C} \geq b+\frac{a}{4C}\]
holds for $C=1$ with $K\geq\lceil\frac{4b}{a}\rceil-3$. For $C=(n-1)$ we need $K\geq\lceil\frac{4(n-1)b}{a}\rceil-4n-3$.
\end{remark}

\begin{thm}\label{uniqshunt2}
Let $m,n\in \N$, $n,m\geq 2$ and $F\colon \R^n_+ \to \mathbb{S}_m $ be continously differentiable, convex and monotonically non-increasing. If 
\[F'\left(z_{j,k}\right)d_j\npreceq 0 \quad \textrm{for all}\quad k\in \lbrace 2,\ldots,K\rbrace,~ j\in \lbrace 1,\ldots,n \rbrace, \]
where
\[
z_{j,k}:=\frac{a}{2}e_j' + \left(a+k\frac{a}{4}\right)  e_j\in \R^n_+, ~ d_j:=\frac{2b-a}{a}e_j' - \frac{1}{2}   e_j\in \R^n
\]  
and  $K:=\max\left(\lceil\frac{4b}{a}\rceil-3,2\right)$, then the following holds:
\begin{itemize}
\item[(a)] $F'(x)\in \mathcal{L}\left( \R^n, \mathbb{S}_m \right)$ is injective for all $x\in [a,b]^n$ and for all $0\neq d\in \R^n$,
\[ \frac{\Vert F'(x)d\Vert_2}{\Vert d\Vert_\infty}\geq \lambda := \min_{j=1,\ldots,n \atop k=2,\ldots,K}\lambda_{\max} \left(F'(z_{j,k})d_j\right)>0. \]
\item[(b)] $F\colon \R^n_+ \to \mathbb{S}^m $ is injective and
\[ \Vert F(x_1)-F(x_2) \Vert_2 \geq \lambda \Vert x_1-x_2 \Vert_\infty \]
for all $x_1,x_2\in [a,b]^n$ i.e. the inverse problem
\begin{equation*}
\textrm{determine } \hat{x} \in [a,b]^n \textrm{ from the knowledge of } F(\hat{x}).
\end{equation*}
is uniquely solvable.
\end{itemize}
\end{thm}
\begin{proof}
Follows directly from Lemma \ref{lem11} and Corollary \ref{coro2}.
\end{proof}
This completes the proof of Theorem \ref{uniqshunt}, since the constructed forward map $\mathcal{F}$ is continously differentiable, convex and monotonically non-increasing by Lemma \ref{monoton}. Note that if $\mathcal{F}$ fulfills the assumptions of Theorem \ref{mainsemidefishunt}, then $\mathcal{F}$ is already injective, i.e. the inverse problem of reconstructing ${\hat \gamma}$ from $\mathcal{F}({\hat \gamma})$ is uniquely solvable, since Theorem \ref{uniqshunt2} holds for both $C=1$ and for $C=(n-1)$. 

\section{Numerical results}
\label{sec:experiments}

In Section \ref{sec:main}, we established a criterion to verify whether a given number of electrodes is sufficient to guarantee the unique solvability of the inverse Robin problem. Additionally, we reformulated the reconstruction problem as a convex semidefinite optimization problem.

We have managed to overcome the non-convexity of the natural data fitting approach, as illustrated in Figure \ref{fig2444}. For a given resolution and electrodes of fixed sizes and positions, \cref{criterion1} can be used to verify unique solvability, while \cref{criterion2} ensures an equivalent convex reformulation. However, there is no general guarantee that these criteria will be met with a finite number of electrodes. It remains to be numerically verified that Criteria \cref{criterion1} and \cref{criterion2}
\begin{equation*}
\lambda =\min_{j=1,\ldots,n \atop k=2,\ldots,K}\lambda_{\max}\left(\mathcal{F}'\left(z_{j,k}\right)d_j\right)>0
\end{equation*}
hold, with $z_{j,k}$ and $d_j$ chosen accordingly for all $ k\in \lbrace 2,\ldots,K\rbrace,~ j\in \lbrace 1,\ldots,n \rbrace$ as in Theorem \ref{uniqshunt} and Theorem \ref{mainsemidefishunt}. For this purpose, we consider a simple geometry of a small circle in a large circle.  The boundary $\Gamma$ is partitioned into $\Gamma_1 \cup \ldots  \cup \Gamma_n$ with equidistant segments, and electrodes $\mathcal{E}_1,\ldots,\mathcal{E}_m$ are uniformly placed along the boundary, as illustrated in Figure \ref{fig44}.\\

Note that  that Criteria \cref{criterion1} and \cref{criterion2} provide upper bounds for the inverse stability constant $1/\lambda$. Once uniqueness and global convergence are established, efforts can be made toward further improving the stability constant $\lambda$. In general, the bound obtained from Theorem \ref{uniqshunt} is more favorable than that from Theorem \ref{mainsemidefishunt}, as Theorem \ref{mainsemidefishunt}  considers more evaluation points than Theorem \ref{uniqshunt}. As illustrated in Figure \ref{fig34}, the value of $\lambda$ improves by taking more electrodes. 
 Nevertheless the ill-posedness of the inverse problem remains, so that $\lambda$ goes to $0$ as the resolution dimension $n$ increases.

\subsection{Implementation of the forward operator}

Motivated by \cite{harrach2021introduction} we discretize the forward map by calculating the element stiffness matrices $\mathcal{B}_0,\mathcal{B}_1,\ldots,\mathcal{B}_n \in \mathcal{S}^D$ and element load vectors $\mathrm{y}_1,\ldots,\mathrm{y}_m\in \R^D$, where $D$ is the dimension of the linear span $<\Lambda_1 ,\ldots,\Lambda_D > \subset H_\square^1(\Omega)$, the element stiffness matrices are given by
 \[ \mathcal{B}_0 = \left( \int_\Omega \sigma \nabla \Lambda_i \nabla \Lambda_j\dx \right)_{i,j=1,\ldots,D}, \quad \mathcal{B}_k = \left( \int_{\Gamma_k} \Lambda_i ~ \Lambda_j\ds\right)_{i,j=1,\ldots,D} \]
 for $k=1,\ldots,n$ and the element load vectors are given by
 \[ \mathrm{y}_k = \left( \sum_{j=1}^m \left(e_k\right)_j  \Lambda_i \big|_{\mathcal{E}_j}\right)_{i=1,\ldots,D} = \left(  \Lambda_i \big|_{\mathcal{E}_k}\right)_{i=1,\ldots,D} \]
for $k=1,\ldots,m$. Note that the chosen basis functions $\Lambda_i$ are constant along the electrodes, since $< \Lambda_1 ,\ldots,\Lambda_D >$ has to be a subspace of $H_\square^1(\Omega)$.
With $P:=\left(\mathrm{y}_1,\ldots,\mathrm{y}_m\right)$ the discretized forward map is given by
\[ \mathcal{F}\colon \R^n_+ \to \mathbb{S}^m_+, \qquad \gamma \mapsto P^T\left(  \mathcal{B}_0 + \sum_{i=1}^n \gamma_i\mathcal{B}_i \right)^{-1}P  \]
where the matrix $\left(\mathcal{B}_0 + \sum_{i=1}^n \gamma_i\mathcal{B}_i \right) $ is symmetric and positive definite, since the bilinear form $b_\gamma$ is symmetric and coercive. Note that the discretized forward map remains to be infinitely often differentiable with 
\begin{equation*}
 \mathcal{F}'(\gamma)\delta =  -\sum_{i=1}^n \delta_i P^T \left(  \mathcal{B}_0 + \sum_{j=1}^n \gamma_j\mathcal{B}_j \right)^{-T} B_{i} \left(  \mathcal{B}_0 + \sum_{j=1}^n \gamma_j\mathcal{B}_j \right)^{-1}P
\end{equation*}
by \cite[Lemma 3]{harrach2021introduction}. Moreover  $\mathcal{F}$ is 
monotonically non-increasing and convex by \cite[Lemma 4]{harrach2021introduction}.

\subsection{Condition for uniqueness and Lipschitz stability} 
In this section, we check for different resolutions whether there is a sufficient number of electrodes such that Condition \cref{criterion1} is satisfied. For all of our exemplary calculations we use the geometry of a small circle inside a large circle with an equidistant partition $\Gamma=\Gamma_1 \cup \ldots  \cup \Gamma_n$ and uniformly placed electrodes $\mathcal{E}_1,\ldots,\mathcal{E}_m$ at the boundary as in Figure \ref{fig44}. We compute 
\[\lambda =\min_{j=1,\ldots,n \atop k=2,\ldots,K}\lambda_{\max} \left(\mathcal{F}'(z_{j,k})d_j\right) \]
\begin{figure}[h]
    \centering
  {{\includegraphics[width=0.44\textwidth]{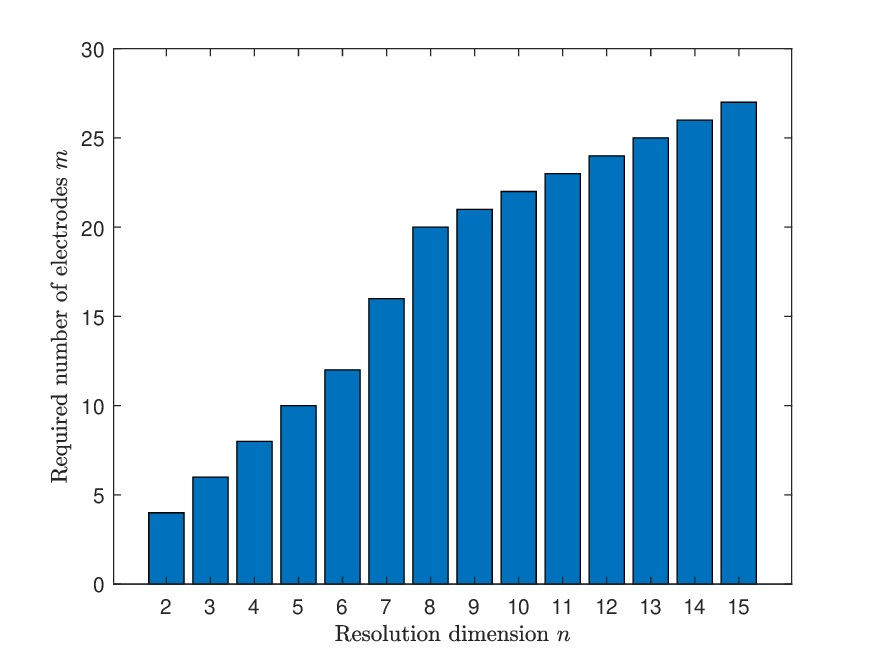} }}%
    \qquad
   {{\includegraphics[width=0.44\textwidth]{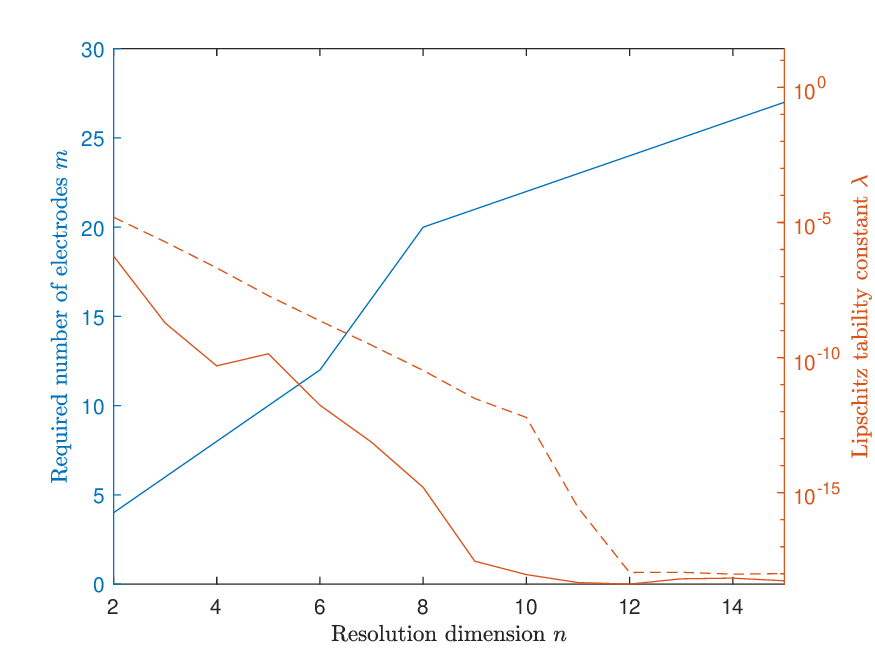} }}%
    \caption{Plot of the required number of electrodes $m$ for Criterion \cref{criterion1} and the stability constant $\lambda$ for different resolution dimensions $n$ with $a=1$ and $b=3$. The dashed line describes $\lambda$ calculated with five extra electrodes.}%
    \label{fig:example2}%
\end{figure}
with $z_{j,k},d_j$ for $ k\in \lbrace 2,\ldots,K\rbrace,~ j\in \lbrace 1,\ldots,n \rbrace$ chosen accordingly to Criterion \cref{criterion1} and check if $\lambda$ is positive. In Figure \ref{fig:example2}, we depict the smallest number electrodes so that $\lambda$ is positive and in particular Criterion \cref{criterion1} is satisfied. We therefore achieve unique reconstruction of the corrosion parameter and Lipschitz stability, i.e.
\[ \Vert \mathcal{F}(\gamma_1)-\mathcal{F}(\gamma_2) \Vert_2 \geq \lambda \Vert \gamma_1-\gamma_2 \Vert_\infty \qquad \textrm{for all} ~ \gamma_1,\gamma_2 \in [a,b]^n. \]

In Figure \ref{fig:example2}, we also compute $\lambda$ for the required number of electrodes and additionally with five extra electrodes to improve stability. However, as shown in Figure \ref{fig34}, for a fixed resolution, the stability constant $\lambda$ increases with the number of electrodes $m$.

\subsection{Condition for convex reformulation}

\begin{figure}[h]%
    \centering
   {{\includegraphics[width=0.44\textwidth]{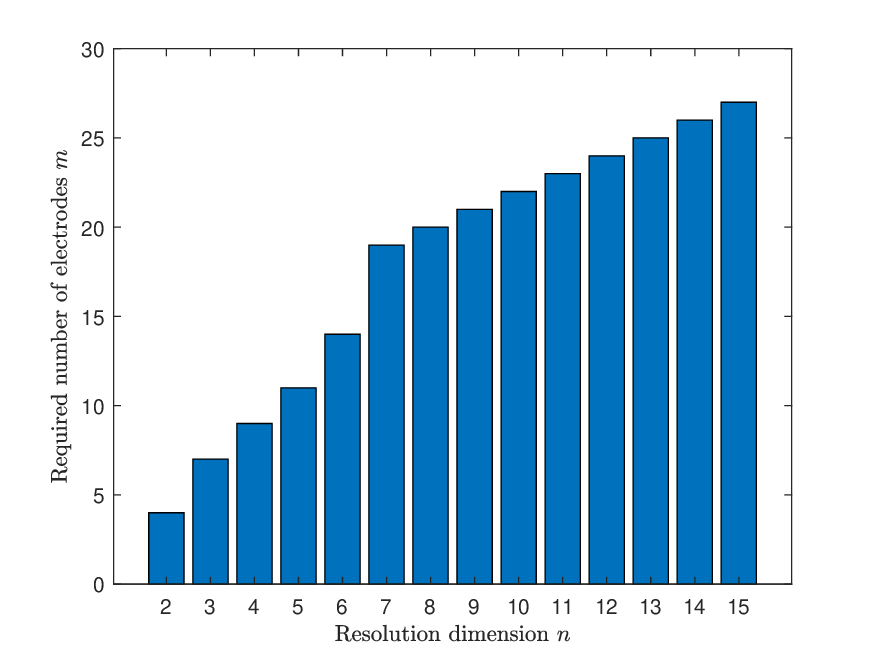} }}
    \qquad 
   {{\includegraphics[width=0.44\textwidth]{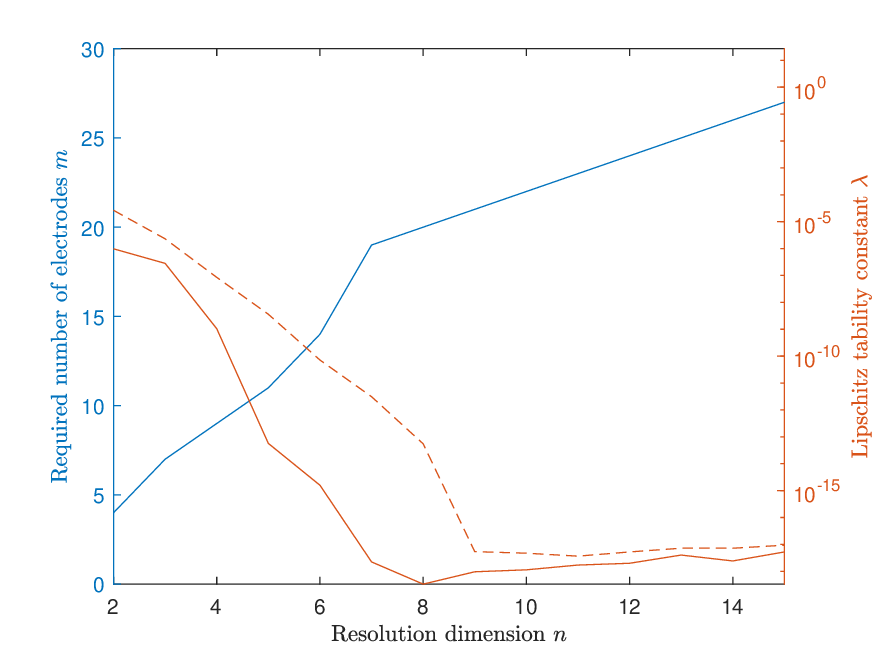} }}%
    \caption{Plot of the required number of electrodes $m$ for Criterion \cref{criterion2} and the stability constant $\lambda$ for different resolutions dimensions $n$ with $a=1$ and $b=3$. The dashed line describes $\lambda$ calculated with five extra electrodes.}%
    \label{fig24}%
\end{figure}

Uniqueness and Lipschitz stability do not guarantee the existence of a globally convergent reconstruction algorithm. To establish the convex semidefinite reformulation, it is necessary to verify Condition \cref{criterion2} from Theorem \ref{mainsemidefishunt} under the more restrictive assumptions. Figure \ref{fig24} illustrates the number of electrodes needed to satisfy Criterion \cref{criterion2} for a given resolution, along with the corresponding value of $\lambda$. We achieve a convex reformulation
\begin{equation}\label{semifinal}
\texttt{minimize} ~ \sum_{i=1}^n \gamma_i~ \textrm{subject to} ~ \gamma \in [a,b]^n, ~ \mathcal{F}(\gamma) \preceq \hat Y=\mathcal{F}(\hat{\gamma})
\end{equation}
\begin{figure}[h]
\begin{center}
\includegraphics[width=0.5\textwidth]{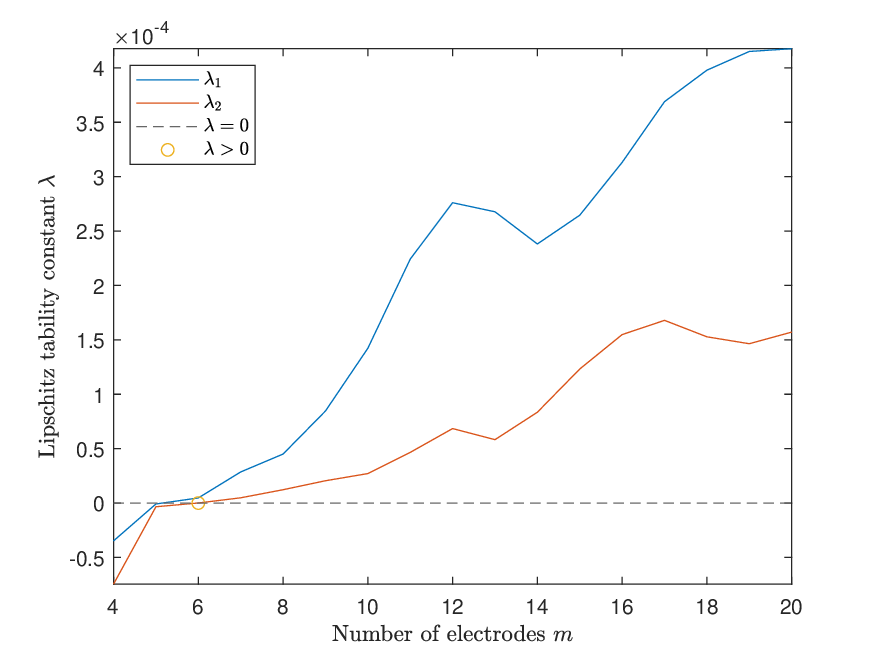}\\
\caption{The constants $\lambda_1$ from Theorem \ref{uniqshunt} and $\lambda_2$ from Theorem \ref{mainsemidefishunt} with a fixed resolution dimension $n=4$ and bounds $a=1$, $b=3$ for an increasing number of electrodes $m$.}\label{fig34}
\end{center}
\end{figure} 
of the inverse problem \cref{inverseproblem}. In this case, the inverse stability constant $\lambda$ decreases even more rapidly than in Condition \cref{criterion1}. This behavior is anticipated since Theorem \ref{mainsemidefishunt} requires verifying the criterion at more evaluation points.

In the setting of Theorem \ref{mainsemidefishunt}, the constant $\lambda$ is also growing with an increasing number of electrodes $m$, as illustrated in Figure \ref{fig34}. Notably, the value of $1/\lambda$ provides an upper bound for the inverse stability constant, since
\[ \Vert \mathcal{F}(\gamma_1)-\mathcal{F}(\gamma_2) \Vert_2 \geq \lambda \Vert \gamma_1-\gamma_2 \Vert_\infty \qquad \textrm{for all}~\gamma_1,\gamma_2\in [a,b]^n. \]

 Since Theorem \ref{mainsemidefishunt} requires evaluating the criterion at more points, the constant $\lambda$ obtained from Criterion \cref{criterion1} generally provides a sharper bound for the stability constant. However, it is important to note that the inverse stability constant, $1/\lambda$, inevitably tends to infinity as the resolution increases, regardless of the method used for its computation. This limitation is an inherent consequence of the ill-posedness of the inverse problem, which remains a fundamental challenge in the study of nonlinear inverse  coefficient problems. 
\begin{figure}[h]%
    \centering
    {{\includegraphics[width=0.42\textwidth]{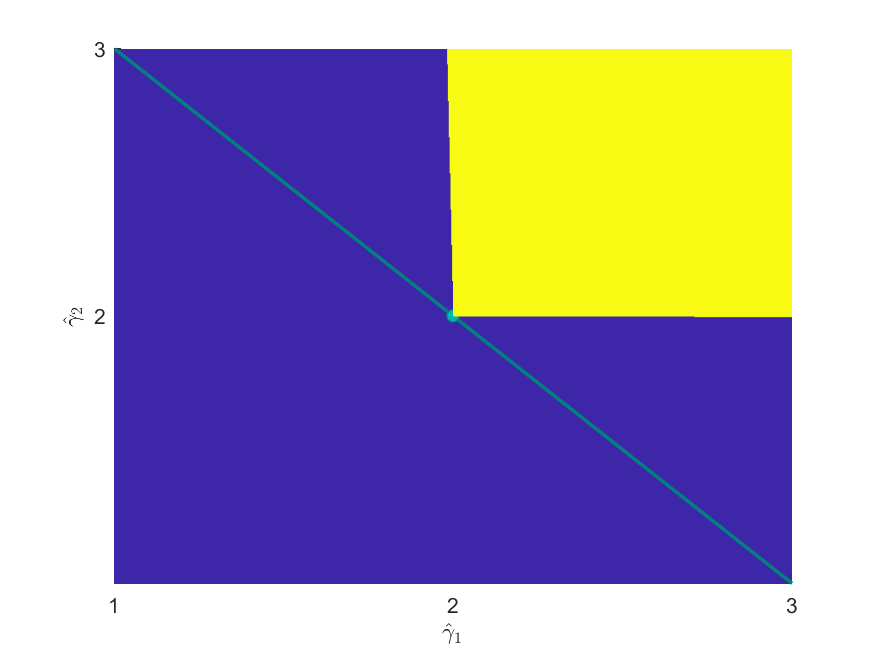} }}%
    \quad
    {{\includegraphics[width=0.42\textwidth]{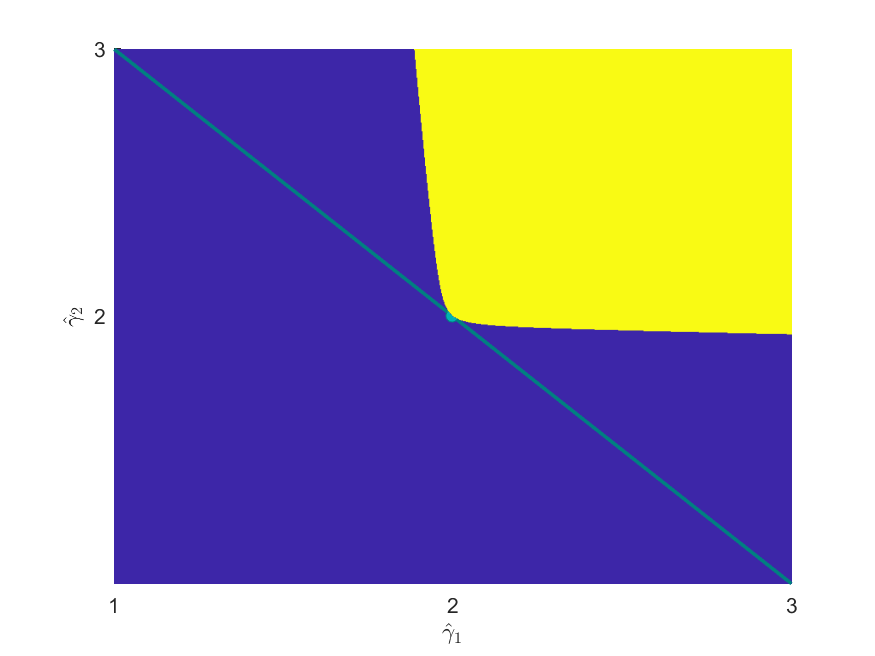} }}%
   \quad
 {{\includegraphics[width=0.42\textwidth]{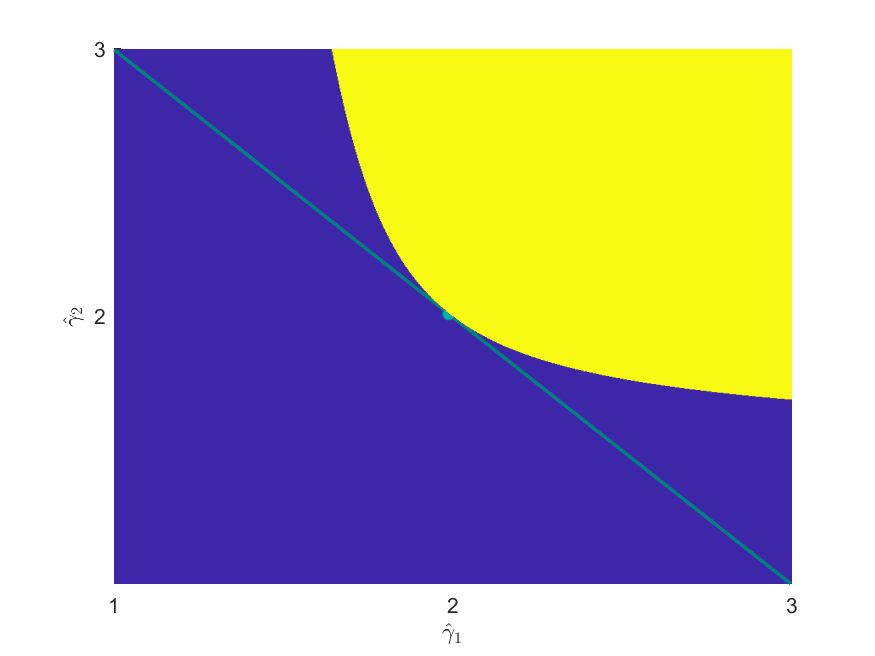} }}%
   \quad
 {{\includegraphics[width=0.42\textwidth]{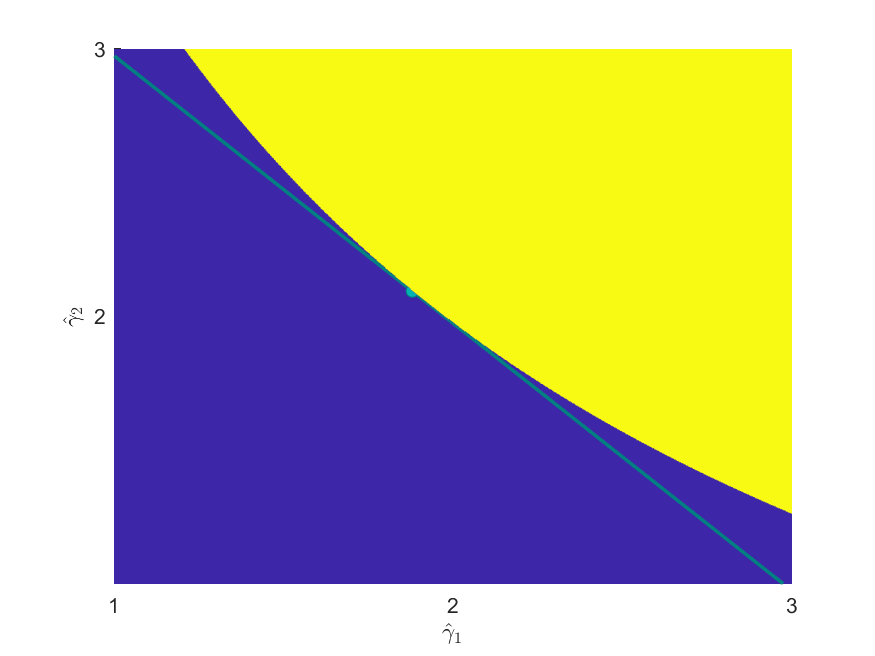} }}
    \caption{Plot of the admissible set without and with noise.}%
    \label{admnoise}%
\end{figure}  
\begin{exmp} In Figure \ref{admnoise}, we plot the admissible set
\[ \lbrace \gamma \in [a,b]^2 ~:~\mathcal{F}(\gamma)\preceq \hat Y\rbrace \]
for the convex problem \cref{semifinal}, considering the parameters $a=1$, $b=3$, $\hat \gamma=(2,2) $, a resolution dimension of $n=2$, and $m=4$ tested electrodes. Additionally, we add random noise to  $\hat Y=\mathcal{F}(\hat \gamma)$ such that $\Vert Y^\delta -\hat Y \Vert_{2} \leq \delta$, and compute
\[\lbrace \gamma \in [a,b]^2 ~:~\mathcal{F}(\gamma)\preceq Y^\delta+\delta I  \rbrace 
 \]
for $\delta=0,10^{-4},10^{-3},10^{-2}$.
\end{exmp}

\subsection{Numerical solver} 
We consider the discretized non-linear inverse problem
\begin{equation}\label{discreteinverseproblem}
\textrm{ reconstruct } \quad \hat{\gamma} \in \R_{+}^n \quad \textrm{ from } \quad \hat{Y}:=\mathcal{F}(\hat{\gamma}),
\end{equation}
where the discretized forward map $\mathcal{F}\colon \mathbb{R}_{+}^n \rightarrow \mathbb{S}_{+}^m \subset \mathbb{R}^{m \times m}$ is given by
$$
\mathcal{F}(\gamma)=P^T\left(  \mathcal{B}_0 + \sum_{i=1}^n \gamma_j\mathcal{B}_j \right)^{-1}P
$$
with $\mathcal{B}_j \in \mathbb{S}^D, j=1, \ldots, n$, and $P \in \mathbb{R}^{D \times m}$ defined as in Section \ref{sec:experiments}. We claim to have a global reconstruction algorithm by reformulating the inverse problem to a minimization problem with a linear objective function and constrains that are convex in the Loewner sense.

\begin{lemma} Let

$$
B_j:=\left(\begin{array}{cc}
\mathcal{B}_j & 0 \\
0 & 0
\end{array}\right) \in \R^{(D+m) \times(D+m)}, \quad \textrm{ and } \quad Y:=\left(\begin{array}{cc}
\mathcal{B}_0 & P \\
P^T & \hat{Y}
\end{array}\right) \in \mathbb{R}^{(D+m) \times(D+m)} .
$$

Then

$$
\mathcal{F}(\gamma) \preceq \hat{Y} \quad \textrm{ if and only if } \quad \gamma_1 B_1+\ldots+\gamma_n B_n+Y \succeq 0 .
$$

\end{lemma}

\begin{proof}

 The assertion follows from the fact that

$$
\hat Y-P^T\left(\mathcal{B}_0+ \sum_{j=1}^n \gamma_j \mathcal{B}_j\right)^{-1} P \quad \textrm{ is the Schur complement of } \quad\left(\begin{array}{cc}
\left(\mathcal{B}_0+\sum_{j=1}^n \gamma_j \mathcal{B}_j \right) & P \\
P^T & \hat Y
\end{array}\right)
$$
and that $\left(\mathcal{B}_0 + \sum_{j=1}^n \gamma_j\mathcal{B}_j \right) \succeq 0$.
\end{proof}

So the convex Problem
\begin{equation*}
\texttt{minimize} ~ \sum_{i=1}^n \gamma_i~ \textrm{subject to} ~ \gamma \in [a,b]^n, ~ \mathcal{F}(\gamma) \preceq \hat Y
\end{equation*}
can be rewritten as
\begin{equation}\label{lineareqn}
\texttt{minimize} ~ \sum_{i=1}^n \gamma_i~ \textrm{subject to} ~ \gamma \in [a,b]^n, ~ \gamma_1 B_1+\ldots+\gamma_n B_n \succeq -Y.
\end{equation}

Problem \cref{lineareqn} can now be solved using standard semidefinite programming methods. Notably, we have established a formal equivalence, meaning that solving the discretized inverse problem \cref{discreteinverseproblem} is mathematically equivalent to solving the corresponding semidefinite program. Consequently, the solutions obtained from both formulations are identical.

\begin{figure}[h]
\begin{center}
\includegraphics[width=0.6\textwidth]{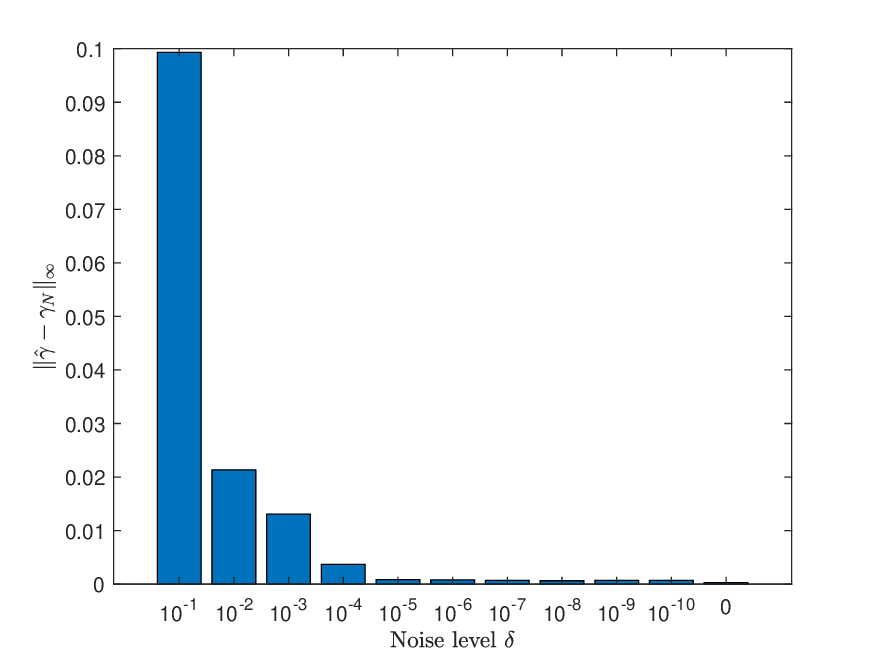}\\
\caption{Reconstruction results using the equivalent semidefinite program for a resolution dimension $ n=20$, bounds $a=1$, $b=3$, number of electrodes $ m=30$ and noise levels $\delta= 10^{-10}, 10^{-1},....,10^{-1},0$.}\label{SDPreconstruction}
\end{center}
\end{figure}

\begin{exmp}\label{example}
For the example of a small circle in a large circle, as depicted in Figure \ref{fig44}, with a resolution dimension of $ n=20$ and bounds $a=1$, $b=3$, the stability constant for $ m=30$ electrodes, computed according to Theorem \ref{uniqshunt}, satisfies $ \lambda \leq 1.3\cdot 10^{-14}$. Similarly, for Theorem \ref{mainsemidefishunt}, we obtain $\lambda \leq 1.2\cdot 10^{-15}$. Despite this stability constant being close to machine precision, the semidefinite program \cref{lineareqn} with $ d=1081$ elements successfully reconstructs the parameter $\hat \gamma$ with a precision of $ \Vert \hat\gamma - \gamma_N\Vert_\infty \leq 3\cdot 10^{-4} $. Even when random noise is added to the measurement data $\hat Y$, the perturbed problem
\begin{equation*}
\texttt{minimize} ~ \sum_{i=1}^n \gamma_i \quad \textrm{subject to} \quad \gamma \in [a,b]^n, ~ \mathcal{F}(\gamma) \preceq Y^\delta +\delta I,
\end{equation*}
with a noise level of $ \Vert\hat Y - Y^\delta\Vert_2 \leq \delta $, maintains an accuracy that is far better than the error estimate provided by Theorem \ref{mainsemidefishunt}, as depicted in Figure \ref{SDPreconstruction}.
\end{exmp}

\section{Conclusion}
\label{sec:conclusions}

This work presents the first explicit characterization of the required number of electrodes necessary to guarantee a desired resolution in reconstructing an unknown Robin transmission coefficient. Previous results have primarily been limited to theoretical uniqueness proofs, whereas our approach provides a computable criterion that ensures both uniqueness and global convergence in a nonlinear inverse problem.

By translating recent advances in global reconstruction techniques to the realistic electrode model, we derive an explicit and verifiable criterion for unique solvability. Once a suitable geometry has been chosen, with electrodes of fixed sizes and positions placed at the boundary, the forward problem can be simulated. By performing a finite number of forward evaluations, Conditions \cref{criterion1} and \cref{criterion2} can be verified numerically by calculating a value $\lambda \in \mathbb{R}$ and checking if it is positive. If $\lambda > 0$, the problem can be rewritten as a uniquely solvable convex semidefinite program, where the stability constant of the inverse problem is given by $1/\lambda$. In cases where the conditions are not satisfied (i.e., $\lambda \leq 0$), additional electrodes may be applied, and the procedure can be repeated until a sufficient number of electrodes is reached.

The numerical experiments in Figures \ref{fig:example2}, \ref{fig24} and \ref{fig34} confirm that the proposed criteria are satisfied when a sufficient number of electrodes is present. While the parameter $\lambda$ reaches machine precision at relatively low resolutions, as shown in Figure \ref{fig34}, the reconstruction remains stable even at significantly higher resolutions, as demonstrated in Example \ref{example}. This stability can be attributed to the self-regularizing nature of the minimization problem, which involves minimizing $\sum_{i=1}^n \gamma_i$, that is the $l_1$-Norm of $ \gamma $, given that $\gamma$ is positive. Moreover, further refinements may be achievable, as the given criteria provide only an upper bound for the stability constant.

Despite successfully finding an equivalent formulation of the nonlinear inverse problem as a semidefinite problem, the fundamental ill-posedness remains a major challenge. Future research could focus on improving the stability bounds, optimizing electrode configurations, and extending the approach to the Calderón problem.

\section*{References}


\end{document}